\documentclass{article}
\usepackage{amscd,amsfonts,euscript,amsmath,amsxtra,amssymb,amsthm}
\usepackage[english]{babel}
\newtheorem{theorem}{Theorem}
\newtheorem{lemma}{Lemma}

\newtheorem{proposition}[theorem]{Proposition}

\theoremstyle{definition} \newtheorem{example}{Example}
\newtheorem{remark}{Remark}

\theoremstyle{remark}

\textwidth13.5cm \numberwithin{equation}{section}

\usepackage[all]{xy}
\def\A{{{\mathbb A}}}

\def\Z{{{\mathbb Z }}}

\def\Spec{{{\rm Spec \,}}}

\def\dim{{{\rm dim \,}}}

\def\ker{{{\rm ker \,}}}

\def\im{{{\rm im \,}}}
\def\id{{{\rm id }}}

\def\length{{{\rm length \,}}}

\def\pr{{{\rm pr }}}
\def\diag{{{\rm diag }}}
\begin{document}

\medskip

\begin{center}
{\Large\sc Fibred product of commutative algebras: generators and
relations}\end{center}
\medskip
\begin{center}
Nadezda V. TIMOFEEVA

\smallskip

P.G. Demidov Yaroslavl' State University

Sovetskaya str. 14, 150000 Yaroslavl', Russia

e-mail: {\it ntimofeeva@list.ru}
\end{center}
\bigskip

\begin{quote} The method of direct computation of universal
(fibred) product in the category of commutative associative
algebras of finite type with unity over a field is given and
proven. The field of coefficients is not supposed to be
algebraically closed and can be of any characteristic. Formation
of fibred product of commutative associative algebras is an
algebraic counterpart of gluing algebraic schemes by means of some
equivalence relation in algebraic geometry. If initial algebras
are finite-dimensional vector spaces the dimension of their
product obeys Grassmann-like formula. Finite-dimensional case
 means geometrically the strict version of adding two collections of
points containing some common part.

The method involves description of algebras by generators and
relations on input and returns similar description of the product
algebra. It is "ready-to-eat"\, even for computer realization. The
product algebra is well-defined: taking another descriptions of
the same algebras leads to isomorphic product algebra. Also it is
proven that the product algebra enjoys universal property, i.e. it
is indeed fibred product. The input data is a triple of algebras
and a pair of homomorphisms $A_1\stackrel{f_1}{\to}A_0
\stackrel{f_2}{\leftarrow}A_2$. Algebras and homomorphisms can be
described in any fashion. We prove  that for computing the fibred
product it is enough to restrict to the case when $f_i,i=1,2$ are
surjective and describe how to reduce to surjective case. Also the
way to choose generators and relations for input algebras is
considered.

Bibliography: 5 titles.

\end{quote}
\section* {Introduction}
\subsection{Motivation and General Problem}
For a motivating example consider affine plane ${\mathbb A}^2$ over
a field $k$ and two non-proportional irreducible polynomials
$g_1,g_2$ in two variables $x,y$ over the same field. Let these
polynom\-ials have zero sets $Z(g_1)\subset {\mathbb A}^2 \supset
Z(g_2)$ and a set of common zeros $Z(g_1,g_2)=Z(g_1)\cap Z(g_2)$. If
$k$ is algebraically closed $Z(g_i)$ are irreducible curves for any
$g_i, i=1,2,$ and $Z(g_1,g_2)$ is a discrete collection of points. A
union of curves $Z(g_1)\cup Z(g_2)$ is given by zeros of the product
polynomial $g_1g_2$, i.e. $Z(g_1)\cup Z(g_2)=Z(g_1g_2)$ and provides
a simple example of amalgamated sum
$Z(g_1)\coprod_{Z(g_1,g_2)}Z(g_2)$ of algebraic schemes $Z(g_i),
i=1,2$ (precise definition will be given below). Transferring to
algebraic counterpart one has $k$-algebras $A_i=k[x,y]/(g_i), i=1,2$
corresponding to curves $Z(g_i), i=1,2$ respectively, and the
$k$-algebra of intersection locus $A_0=k[x,y]/(g_1, g_2)$. There are
obvious $k$-algebra homomorphisms $f_i:A_i \twoheadrightarrow A_0$.
Formation of the algebra $k[x,y]/(g_1g_2)$ for the union of two
components provides an example of universal (fibred) product of
$k$-algebras \linebreak $A_1 \times_{A_0}A_2$.

For our purposes it is enough to restrict by affine algebraic
$k$-schemes of finite type, i.e. prime spectra (sets of prime
ideals with natural topology and collection of local rings forming
a structure sheaf on the spectrum) of commutative associative
$k$-algebras of finite type with unity (for complete theory cf.
\cite[ch.2]{Hart}). We focus on algebraic side and operate
exclusively with underlying algebras.

We describe general {\it quotient problem} in the category of
algebraic schemes over a field $k$. The problem includes following
ingredients:
\begin{itemize}
\item{a scheme $X$;}
\item{a subscheme $R \subset X\times X$;}
\item{morphisms $p_i: R \subset X\times X \to X$ defined as
composites of the immersion with the projection on $i$th factor.}
\end{itemize}

The subscheme $R \subset X\times X$ is said to be an {\it
equivalence relation}; this means that it satisfies three
requirements as follows:
\begin{enumerate}
\item{$R$ is {\it reflexive}, i.e. $R \supset \diag(X)$ where
$\diag (X)$ is an image of the diagonal immersion $\diag: X
\hookrightarrow X\times X$;}
\item{$R$ is {\it symmetric}, i.e. the immersion
$R \subset X \times X$ is stable under the involution intertwining
first and second factors of the product $X\times X$;}
\item{$R$ is {\it transitive}, i.e.
$\pr_{13}(R \times X \cap X\times R) \subset R$ where the
intersection is taken in $X\times X\times X$. Here $\pr_{13}:
X\times X\times X\to X\times X$ is the projection to the product
of the 1st and 3rd factors.}
\end{enumerate}

The question is \cite[ch. 1, 4.3]{Dan} to construct (if possible)
the {\it universal quotient} $X/R$. This is an object  fitting
into the commutative diagram
\begin{equation*}\xymatrix{R\ar[d]_{p_1} \ar[r]^{p_2}& X \ar[d]\\
X \ar[r]& X/R}
\end{equation*}
such that for any other commutative square
\begin{equation*}\xymatrix{R\ar[d]_{p_1} \ar[r]^{p_2}& X \ar[d]\\
X \ar[r]& T}
\end{equation*}
there is a unique morphism $\tau: X/R \to T$ (in the appropriate
category) fitting into the commutative diagram
\begin{equation*}
\xymatrix{R\ar[d]_{p_1} \ar[r]^{p_2}& X \ar[d] \ar[ddr]\\
X \ar[r] \ar[drr]& X/R \ar[rd]^\tau &\\
&& T}
\end{equation*}
The best result should be to know precisely when the universal
quotient exist in some  category (for example, in the category of
algebraic schemes)  and when it does not.

\subsection{Particular Cases} \begin{enumerate} \item{The notion
of algebraic space \cite{Knu} which appears when morphisms $p_i,
i=1,2$ assumed to be \'{e}tale. In this case the object $X/R$
belongs to the category of algebraic spaces.} \item{Open covering
of a scheme $Y$ when $X=\bigsqcup U_{\alpha}$ is a disjoint union
of open subschemes of $Y$, $R=\bigsqcup R_{\alpha \beta}$ and
morphisms $R_{\alpha \beta} \to U_{\alpha}$ and $R_{\alpha \beta}
\to U_{\beta}$ are local isomorphisms. In this case the object
$X/R$ is the scheme $Y$ whose open cover was considered.}
\item{Let the scheme $X$ is acted
upon by an algebraic group $G$ and $R$ is a subscheme induced by
$G$-equivalence \cite{MumFo}. In this case if $X/G=X/R$ exists its
universality means that it is a categorical quotient of the scheme
$X$ by $G$.} \end{enumerate}

\subsection{Case of Several Components} Now let $X$ be
a disjoint union of schemes $X=X_1\bigsqcup X_2$. Then
$$
X\times X=X_1 \times X_1\bigsqcup X_1 \times X_2 \bigsqcup
X_2\times X_1 \bigsqcup X_2 \times X_2
$$
 and the equivalence
relation is broken into 4 disjoint components
$$
R=R_{11} \bigsqcup R_{12}\bigsqcup R_{21}\bigsqcup R_{22}.
$$
The problem reduces to the search of the universal completion of
the diagram
\begin{equation*} \xymatrix{R_{ij}\ar[d]_{p_i} \ar[r]^{p_j}& X_j\\
X_i}
\end{equation*}
This universal completion is called an {\it amalgam} (or an {\it
amalgamated sum}) of schemes $X_i$ and $X_j$ with respect to
$R_{ij}$ and is denoted as $X_i\coprod_{R_{ij}} X_j$. If for each
pair $i,j=1,2$ the amalgam $X_i\coprod_{R_{ij}} X_j$ exists then
$X/R=X\coprod_R X=\coprod_{i,j}(X_i\coprod_{R_{ij}} X_j)$.

\subsection{Schemes and algebras}
We recall that affine Grothendieck' scheme of finite type over a
field $k$ is a prime spectrum $\Spec A$ of associative commutative
algebra $A$ of finite type with unity over $k$.

 Recall that the $k$-algebra $A$ is said to
be of finite type if it admits a surjective homomorphism of
polynomial algebra in finite set of variables $k[x_1, \dots , x_n]
\twoheadrightarrow A$. In particular this means that  algebras $A$
under consideration  can have any finite Krull dimension (if $A$
admits an epimorphism of polynomial algebra in $n$ variables then
the Krull dimension of $A$ is not greater than $n$). In this case
the algebra $A$ as $k$-vector space can be infinitely-dimensional.

Generally, fibred product and  amalgam (or  product and corpoduct,
or small limits) are dual categorical notions
\cite[1.17,1.18]{GelMan} and are not obliged to exist in any
category. The schemes under consideration are prime spectra of
associative commutative algebras of finite type over $k$. Duality
is done by functorial correspondence taking $k$-algebra to its
prime spectrum. The fibred product of associative commutative
$k$-algebras of finite type as taken by this functor to the
amalgam of corresponding spectra. The functorial behavior of
fibred product of algebras is analogous to one of fibred product
of schemes.

Despite that existence of fibred product for associative
commutative algebras is known, the method of explicit computation
of it is not described in the literature. We fill this gap.

Let $A$ be a commutative associative algebra of finite type with
unity; then it can be represented as a quotient algebra of
$n$-generated polynomial algebra $\phi: k[x_1, \dots ,x_n]
\twoheadrightarrow A$ where $n$ depends on the structure of $A$.
The representing homomorphism $\phi$ corresponds to the immersion
of the scheme $\Spec A$ into $k$-affine space ${\mathbb A}^n=\Spec
k[x_1, \dots ,x_n]$. By means of this immersion we interpret the
abstract
 $k$-scheme $Z$ as a closed subscheme $Z\subset
{\mathbb A}^n$ of affine space. We call the isomorphism $A \cong
k[x_1, \dots , x_n]/\ker \phi$ (resp., the immersion $Z \subset
{\mathbb A}^n$) the {\it affine representation} of the algebra $A$
(resp., scheme $Z$). This is the reason why any associative
commutative algebra of finite type with unity is called an {\it
affine algebra}. Then the $k$-algebra is affine if and only if it
has affine representation. Any affine algebra has many different
affine representations.

We prove the following result.
\begin{theorem} The universal (fibred) product of affine algebras
of finite type over a field can be compute explicitly by means of
generators and relations.
\end{theorem}
The method to compute fibred products of affine algebras
constitutes the main subject of this article and described below.

First we fix the terminology and describe the algorithm which
constructs algebra to be a universal product of algebras by means
of their appropriate affine representations. Then we confirm that
choice of different appropriate affine representations leads to
isomorphic algebras. Finally we prove the universality and hence
confirm that the obtained algebra is indeed a fibred product.

I would like to thank Laboratory of Discrete and Computational
Geometry  for the support during my work on this article (Grant No
1.1875.2014 (No 2014/258)) and my Diploma student I.~Chistousov
for testing the finite-dimensional version of the method by
thoroughly computing various examples.

\section{Construction}
\subsection{Glossary}
We  work with polynomial rings of view $k[x_1, \dots, x_n],$ its
ideals and quotients.

The {\it basis} of the ideal $I$ is any system of generators of
$I$ as $k[x_1, \dots, x_n]$-module.

The basis is {\it minimal} if it does not span $I$ whenever any of
elements is excluded.

The {\it $k$-basis} of polynomial algebra (which is not obliged to
be with unity) is its basis if the algebra is considered as
$k$-vector space.

The {\it $k$-relation} between elements of the algebra $A$ (which
is not obliged to be with unity) is a nontrivial $k$-linear
function (nontrivial $k$-linear combination) in these elements
which equals 0 in $A$. As usually, linear combinations with finite
number of nonzero coefficients are considered.

Fix the natural ordering of  variables $x_1, \dots, x_n$; then we
can use shorthand notation $x^{\alpha}$ for the monomial
$x_1^{\alpha_1} \dots x_n^{\alpha_n}$. The symbol $\alpha$ denotes
the row of degrees $\alpha_1, \dots , \alpha_n$.
\subsection{Algorithm}
\subsubsection*{Reduction to surjective morphisms} We are given three
algebras $A_0, A_1, A_2$ with morphisms
$A_1\stackrel{f_1}{\longrightarrow}
A_0\stackrel{f_2}{\longleftarrow } A_2$. Show that we can assume
that both $f_i$, $i=1,2$ are surjective. Since the morphisms $f_i$
must include into commutative squares of the form
\begin{equation}\label{cd1} \xymatrix{A_0 &\ar[l]_{f_2}
A_2\\
\ar[u]^{f_1} A_1& \ar[l]^{\chi_1} A_T
\ar[u]_{\chi_2}}\end{equation} then for any such a commutative
square the composite morphisms $f_1 \circ \chi_1$ and $f_2 \circ
\chi_2$ have coincident images in $A_0.$ By commutativity of the
square $\im f_1 \circ \chi_1 =\im f_2 \circ \chi_2 \subset \im
f_i$, $i=1,2$. This means that $\im f_i \circ \chi_i \subset \im
f_1 \cap \im f_2$ and the morphism $\chi_i$ factors through the
subalgebra $f_i^{-1}(\im f_1 \cap \im f_2) \subset A_i$ for
$i=1,2$. Then all algebras of interest $A_T$ complete the square
\begin{equation*}
\xymatrix{\im f_1 \cap \im f_2& \ar@{->>}[l]_{f'_2} {f_2}^{-1}(\im
f_1 \cap
\im f_2)\\
\ar@{->>}[u]^{f'_1} f_1^{-1}(\im f_1 \cap \im f_2)& \ar[l] A_T
\ar[u]}
\end{equation*}
to commute.

Denote $A'_1\stackrel{f'_1}{\twoheadrightarrow}A'_0
\stackrel{f'_2}{\twoheadleftarrow} A'_2$ where $A'_0=\im f_1 \cap
\im f_2,$ $A'_i=f_i^{-1}(A'_0)$ and $f'_i=f_i|_{A'_i},$ $i=1,2.$

Let there is a fibred product $A'_1\times _{A_0'} A'_2$ for the
case $A'_1\stackrel{f'_1}{\twoheadrightarrow}A'_0
\stackrel{f'_2}{\twoheadleftarrow} A'_2$. Confirm that $A'_1\times
_{A_0'} A'_2$ is also a universal product for the initial
arbitrary data \linebreak $A_1\stackrel{f_1}{\longrightarrow}A_0
\stackrel{f_2}{\longleftarrow} A_2$. By the construction of
algebras $A'_i$, $i=0,1,2,$ and by the definition of the algebra
$A'_1\times _{A_0'} A'_2$ it includes into the commutative diagram
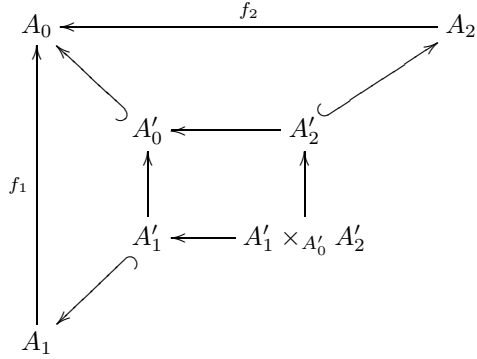
\begin{equation*}\xymatrix{A_0&&&\ar[lll]_{f_2}A_2\\
&\ar@{^(->}[ul]A'_0&\ar[l]A'_2 \ar@{^(->}[ur]
\\
&A'_1 \ar[u] \ar@{^(->}[ld]&\ar[l]A'_1\times _{A_0'} A'_2\ar[u]
\\
A_1\ar[uuu]^{f_1}}
\end{equation*}
Let $A_T$ fit into the commutative diagram (\ref{cd1}). Then by
the universality of $A'_1\times _{A_0'} A'_2$ as a product there
is a unique homomorphism $\varphi: A_T\to A'_1\times _{A_0'} A'_2$
which fits in the commutative diagram
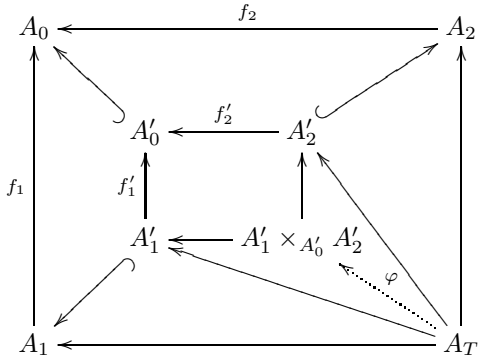
\begin{equation*}\xymatrix{A_0&&&\ar[lll]_{f_2}A_2\\
&\ar@{^(->}[ul]A'_0&\ar[l]_{f'_2}A'_2 \ar@{^(->}[ur]
\\
&A'_1 \ar[u]^{f'_1} \ar@{^(->}[ld]&\ar[l]A'_1\times _{A_0'}
A'_2\ar[u]
\\
A_1\ar[uuu]^{f_1}&&& \ar[lll] \ar[ull] \ar@{.>}[ul]_{\varphi}
\ar[uul] A_T \ar[uuu]}
\end{equation*}
This diagram provides universality of the algebra $A'_1\times
_{A_0'} A'_2$ as a product for initial data, i.e. $$A'_1\times
_{A_0'} A'_2=A_1\times _{A_0} A_2.$$

  Now we can replace everywhere our initial arbitrary data by
the more special case when morphisms $f_i$, $i=1,2$, are
surjective.

\subsubsection*{Easy case} Assume that
$A_i \cong k[x_1, \dots, x_n]/I_i $, $i=0,1,2,$ so that $A_0=A_1
\otimes_{k[x_1, \dots, x_n]} A_2$. This assumption says that the
subscheme $\Spec A_0 \cong Z_0\subset {\mathbb A}^n$ is an
intersection of subschemes $Z_1\subset {\mathbb A}^n$ and $Z_2
\subset {\mathbb A}^n$, $Z_i \cong \Spec A_i$, $i=1,2$.

Choose minimal bases in ideals $I_i$, $i=1,2$. Consider set of
those monomials in $k[x_1, \dots , x_n]$ which  are taken to 0
 in both $A_1$ and $A_2$ by homomorphisms $f_1$ and $f_2$ respectively. It is clear
that if $x^{\alpha}$ is such a monomial that for all $\beta \in
\Z_{\ge 0}^n$ the monomial $x^{\alpha + \beta}$ is taken to 0 in
both $A_1$, $A_2$. Then the set of monomials under consideration
generates an ideal $J$.

 Quotient algebra $k[x_1, \dots ,x_n]/J$
includes into the commutative square
\begin{equation*}\xymatrix{A_0 &\ar@{->>}[l]_{f_2}
A_2\\
\ar@{->>}[u]^{f_1} A_1& \ar@{->>}[l] k[x_1, \dots ,x_n]/J
\ar@{->>}[u]}\end{equation*}

Consider $k$-basis elements in $k[x_1, \dots ,x_n]/J$. Since $A_i$
are not obliged to be finite-dimensional over $k$, $k[x_1, \dots
,x_n]/J$ can  also have infinite $k$-basis. We form the list $L$
of $k$-relations among $k$-basis elements of $k[x_1, \dots
,x_n]/J$ as follows.

 The
$k$-relation $\sum a_{\alpha} x^{\alpha}$ between elements of
$k[x_1, \dots ,x_n]/J$ (possibly monomial) is include into $L$ if
and only if it equals 0 in both of $A_1$, $A_2$. It is clear that
the linear span $<L>$ is an ideal in $k[x_1, \dots ,x_n]/J.$

Now set $A=(k[x_1, \dots ,x_n]/J)/<L>=k[x_1, \dots ,x_n]/(J+<L>)$.
\begin{remark} By Hilbert's basis theorem, ideals $J$, $<L>,$ and
$J+<L>$ admit finite bases.
\end{remark}
\begin{remark} If $A_1$ and $A_2$ finite-dimensional then $\dim A_1\times_{A_0}
A_2=\dim A_1+\dim A_2 -\dim A_0$. This follows immediately from
the algorithm described. If $k$ is algebraically closed then
$\length A_i=\dim A_i$, $i=0,1,2,$ and then $\length A_1 \times
_{A_0} A_2 = \length A_1 + \length A_2 -\length A_0.$
\end{remark}

\subsubsection*{Hard case} How to build up affine representations
of algebras $A_0,A_1,A_2$ to validate the Easy case, i.e. such
that $A_0=A_1 \otimes_{k[x_1, \dots, x_n]} A_2$?

Start with principal ideals in $A_0$: only those which are maximal
under inclusion are necessary. Take a generator of each such an
ideal. Choose any maximal $k$-linearly independent subset in the
set of generators chosen. Since $A_0$ is an algebra of finite
type, this subset is obliged to be finite; let it consist of $s$
elements $g_1, \dots, g_s$. Each element $g_j$  corresponds to the
variable $x_j$ in $k[x_1, \dots, x_s].$ The construction done
fixes a homomorphism of $k$-algebras $h_0:k[x_1, \dots,
x_s]\twoheadrightarrow A_0$. Let $K_0:=\ker h_0$. Since $k[x_1,
\dots ,x_n]$ in Noetherian then $K_0$ is  finitely generated
ideal.

Let $g_j'$ be one of preimages  of the element  $g_j\in A_0$ in
$A_1$, $g''_j$ be one of preimages of the same element in $A_2$.
We complete the set  $g'_1, \dots ,g'_s$ to form the set of
$k$-linearly independent generators of the algebra $A_1$, by
adding elements  $g'_{s+1}, \dots. g'_m \in A_1$. Similarly, the
set $g''_1, \dots , g''_s$ is completed to the set of $k$-linearly
independent generators of the algebra $A_2$ by adding elements
$g''_{m+1}, \dots, g''_n$. We put the variables $x_{s+1},
\dots,x_m$ in the correspondence to the elements $g'_{s+1}, \dots
,g'_{m}$ and the variables $x_{m+1}, \dots, x_n$ to the elements
 $g''_{m+1}, \dots, g''_n$.

This leads to  homomorphisms $h_i:k[x_1,
\dots,x_n]\twoheadrightarrow A_i$, $i=0,1,2$, $n\ge s$ defined by
following correspondences
\begin{eqnarray*}&&h_0(x_j)=g_j,\;\; j=1,\dots,s,\\
&&h_0 (x_j)=0,\;\;j=s+1, \dots , n,\\
&&h_1(x_j)=g'_j,\;\; j=1,\dots,m,\\
&&h_1(x_j)=0,\;\; j=m+1,\dots,n,\\
&&h_2(x_j)=g''_j,\;\; j=1, \dots,s,m+1, \dots,n,\\
&&h_2(x_j)=0,\;\; j=s+1,\dots,m.
\end{eqnarray*}
In particular,  $h_0=f_i\circ h_i$.

 Affine representations of
algebras $A_0, A_1, A_2$ and their homomorphisms $f_1, f_2$ lead
to closed immersions of their spectra into affine space ${\mathbb
A}^n$ according to the following commutative diagram
\begin{equation*}\xymatrix{{\mathbb A}^n&\ar@{_(->}[l]_{h_2^{\sharp}} \,\Spec A_2\\
\ar@{_(->}[u]^{h_1^{\sharp}} \Spec
A_1&\ar@{_(->}[l]^{f_1^{\sharp}} \ar@{_(->}[ul]_{h_0^{\sharp}}
\,\Spec A_0 \ar@{_(->}[u]_{f_2^{\sharp}}}
\end{equation*}
By some modification of representing homomorphisms  $h_i$,
$i=0,1,2,$ we will easily achieve that in the appropriate affine
space
$$h_0^{\sharp}(\Spec A_0)=h_1^{\sharp}(\Spec A_1) \cap
h_2^{\sharp}(\Spec A_2).$$ For this purpose choose any system of
generators $\kappa_1,\dots,\kappa_r$ of the ideal $K_0$ (as
$k[x_1, \dots, x_n]$-module). Add to the set of variables $x_1,
\dots, x_n$ additional variables (whose number equals to the
number  $r$ of generators if $K_0$ chosen) $y_1, \dots, y_r$, and
consider in  $A_1$-algebra $A_1[y_1, \dots, y_r]$ an ideal $K'_0$
generated by all relations of view $y_l-\kappa_l(g'_1, \dots,
g'_m),$ $l=1,\dots, r$. Then we have a commutative diagram
\begin{equation*}\xymatrix{
k[x_1, \dots, x_n, y_1, \dots,y_r] \ar@{->>}[d]_{h'_1}
\ar@{->>}[rd]^>>>>>>>>{h'_0} \ar@{->>}[r]^>>>>>{h'_2}&A_2
\ar@{->>}[d]^{f_2}\\A_1[y_1, \dots y_r]/K'_0 \ar@{->>}[r] &A_0}
\end{equation*}
There $h'_i (x_j)=h_i(x_j)$ for $j=1,\dots ,n$ and for $i=0,1,2$,
but $h'_i(y_j)=0$ for $i=0,2$ and $j=1,\dots, r$.
\begin{lemma} There is an isomorphism of $k$-algebras $A_1[y_1, \dots y_r]/K'_0 \cong
A_1$.
\end{lemma}
\begin{proof}
Let $K_1=\ker h_1$. Denote by $J'$ the ideal in the ring  $k[x_1,
\dots, x_n, y_1, \dots, y_r]$ which is generated by relations
$y_l-\kappa_l(x_1, \dots, x_n),$ $l=1,\dots,r$. Obviously, $J'
\subset \ker h'_1$. The isomorphism
 of $k[x_1, \dots, x_n, y_1,\dots,y_r]$-modules $A_1[y_1, \dots
y_r]/K'_0 \cong A_1$ follows from the exact diagram of $k[x_1,
\dots, x_n, y_1, \dots y_r]$-modules
\begin{equation*}\xymatrix{&0 \ar[d]&0 \ar[d]\\
& J'\ar[d]
\ar[r]^<<<<<<<<<{\sim}&(y_1, \dots, y_r)\ar[d]&\\
0\ar[r]&\ker h'_1 \ar[d] \ar[r]&k[x_1, \dots, x_n, y_1,\dots,y_r]
\ar[d] \ar[r]&A_1[y_1, \dots,y_r]/K'_0 \ar[d]^{\wr}
\ar[r]&0\\
0\ar[r]&K_1 \ar[d] \ar[r] &k[x_1, \dots ,x_n] \ar[d] \ar[r]&A_1
\ar[r]&0\\
&0&0}\end{equation*} where the isomorphism  $J' \cong (y_1, \dots,
y_r)$ is an isomorphism of free modules of equal ranks done by the
correspondence $y_l-\kappa_l(x_1, \dots, x_n) \mapsto y_l,$ $l=1,
\dots, r$.
 The $k$-algebra isomorphism  $A_1[y_1, \dots y_r]/K'_0 \cong A_1$
follows from the fact that the quotient algebra $A_1[y_1, \dots
y_r]/K'_0 $ admits the same system of generators  $g'_1, \dots,
g'_m$ as an algebra  $A_1$, with same relations (equal to
generators of the ideal $K_1$).
\end{proof}

\begin{proposition}\label{prop} The homomorphisms $h'_i$ lead to the
expression $$A_0=A_1 \otimes_{k[x_1, \dots, x_n,y_1,\dots ,y_r]}
A_2.$$
\end{proposition}
\begin{proof} To prove the proposition one needs to confirm that
$A_0$ is universal as a coproduct. Let Q be a $k$-algebra supplied
with two homomorphisms $A_1 \stackrel{\varphi_1}{\longrightarrow}
Q \stackrel{\varphi_2}{\longleftarrow} A_2$ such that the
following diagram commutes:
\begin{equation*}
\xymatrix{k[x_1, \dots, x_n,y_1,\dots, y_r] \ar@{->>}[d]_{h'_1}
\ar@{->>}[r]^>>>>>>{h'_2}
& A_2 \ar@{->>}[d]_{f_2} \ar[rdd]^{\varphi_2}\\
A_1 \ar@{->>}[r]^{f_1} \ar[rrd]_{\varphi_1}& A_0&\\
&&Q}
\end{equation*} Commutativity of the ambient contour
guarantees that $\varphi_1(A_1)=\varphi_2(A_2)$ and hence we can
replace Q by $\varphi_1(A_1)=\varphi_2(A_2)$. However we preserve
the notation $Q$ but assume that homomorphisms $\varphi_i$ are
surjective. By the construction of homomorphisms $h'_i,$ the
algebra $A_0$ is generated by the images of first $s$ variables
$x_1, \dots, x_s$.

Then by the commutativity of the ambient contour
$\varphi_1h'_1(x_j)=\varphi_2h'_2(x_j)$ for $j=1,\dots,n$. In
particular, by the construction of homomorphisms  $h_i$ (and of
homomorphisms $h'_i$ built up by means of them),
$\varphi_1h'_1(x_j)=\varphi_2h'_2(x_j)=0$ for $j=s+1, \dots, n$.
For  $j=1,\dots s$ we have $0=\varphi_1\circ h'_1(y_l-
\kappa_l(x_j))=\varphi_2\circ h'_2(y_l-\kappa_l(x_j))=
-\varphi_2h'_2\kappa_l(x_j)$ for $l=1,\dots r$ since
$h'_2(y_l)=0$. This implies that homomorphisms  $\varphi_i$,
$i=1,2,$ factor through the algebra $A_0.$

The homomorphism $\varphi_0: A_0 \to Q$ is uniquely defined on
generators $h'_0(x_j)$ of $A_0$ by the correspondence $h_0(x_j)
\mapsto \varphi_i \circ h'_i (x_j),$ $i=1,2,$ $j=1, \dots, s.$
\end{proof}

\begin{remark} Proposition \ref{prop} means that affine
representations constructed are such that the intersection of
images of schemes $\Spec A_1$ and $\Spec A_2$ under closed
immersions into  $\Spec k[x_1, \dots, x_n, y_1,\dots,y_r]={\mathbb
A}^{n+r}$ is the image of $\Spec A_0.$
\end{remark}

\begin{remark} In further considerations we use the notation
of the view  $$A_1 \stackrel{h_1}{\longleftarrow} k[x_1, \dots,
x_n] \stackrel{h_2}{\longrightarrow} A_2$$ assuming that
homomorphisms  $h_i$. $i=1,2,$ are surjective and that they admit
application of the Easy case, i.e. Proposition \ref{prop} holds
for them.
\end{remark}
\section{Well-Definedness and Universality}
\subsection{Different affine representations chosen}
\begin{proposition} Different affine representations of algebras $A_0,
A_1, A_2$ lead to isomorphic product algebras $A$.
\end{proposition}

\begin{proof} Let we have two pairs of different affine
representations defined by pairs of homomorph\-isms $$A_1
\stackrel{h_1}{\twoheadleftarrow} k[x_1, \dots, x_n]
\stackrel{h_2}{\twoheadrightarrow} A_2 \quad \mbox{ \rm and}\quad
A_1 \stackrel{h'_1}{\twoheadleftarrow} k[x_1', \dots, x_m']
\stackrel{h'_2}{\twoheadrightarrow} A_2$$ such that $A_0=A_1
\otimes_{k[x_1', \dots, x_m']} A_2=A_1 \otimes_{k[x_1, \dots,
x_n]} A_2$. They correspond to two pairs of closed immersions of
prime spectra $$ \Spec A_1 \stackrel
{h_1^{\sharp}}{\hookrightarrow} {\mathbb A}^n
\stackrel{h_2^{\sharp}}{\hookleftarrow} \Spec A_2 \quad \mbox{\rm
and} \quad \Spec A_1 \stackrel {{h_1'}^{\sharp}}{\hookrightarrow}
{\mathbb A}^m \stackrel{{h_2'}^{\sharp}}{\hookleftarrow} \Spec A_2
$$
such that $h_1^{\sharp}(\Spec A_1) \cap h_2^{\sharp}(\Spec
A_2)=h_0^{\sharp}(\Spec A_0)$ and ${h'_1}^{\sharp}(\Spec A_1) \cap
{h'_2}^{\sharp}(\Spec A_2)={h'_0}^{\sharp}(\Spec A_0)$. This leads
to closed immersions $\overline {h_i^{\sharp}}: \Spec A_i
\hookrightarrow {\mathbb A}^n \times {\mathbb A}^m={\mathbb
A}^{n+m}$ as composites with the diagonal immersion $\diag$:
$$\overline {h_i^{\sharp}}:\Spec A_i
\stackrel{\diag}{\hookrightarrow} \Spec A_i \times \Spec A_i
\stackrel{(h_i^{\sharp},
{h_i'}^{\sharp})}{-\!\!-\!\!\!\longrightarrow} {\mathbb A}^n
\times {\mathbb A}^m.$$ Now confirm that in this case also
$\overline{h_1^{\sharp}}(\Spec A_1) \cap
\overline{h_2^{\sharp}}(\Spec A_2)=\overline{h_0^{\sharp}}(\Spec
A_0)$.

Homomorphisms $\overline h_i$ are defined as composites
\begin{eqnarray}
k[x_1, \dots, x_n]\otimes _k k[x_1', \dots, x_m']\stackrel{(h_i,
h_i')}{-\!\!\!\longrightarrow}& A_i \otimes_k A_i
&\stackrel{\delta_i}{\to} A_i, \nonumber
\\
g_1 \otimes g_2 \;\quad \mapsto\, & h_i(g_1) \otimes h_i'(g_2)
&\mapsto h_i(g_1) \cdot h_i'(g_2). \nonumber
\end{eqnarray}

Let $Q$ be an algebra together with two homomorphisms $\varphi_i:
A_i \to Q$ such that $ \varphi_1\circ \overline h_1 = \varphi_2
\circ \overline h_2$. To confirm that $A_0=A_1 \otimes _{k[x_1,
\dots, x_n]\otimes _k k[x_1', \dots, x_m']}A_2$ it is necessary to
construct a unique homomorphism $\varphi_0: A_0 \to A_T$ such that
$\varphi_i=\varphi_0 \circ f_i$. Consider commutative diagrams
\begin{equation*} \xymatrix{A_i\otimes_k A_i \ar[d]_{\delta_i} \ar[r]^{(f_i,
f_i)}& A_0 \otimes _k A_0 \ar[d]^{\delta_0} \ar[ddr]^{\varphi'_0}\\
A_i \ar[r]^{f_i} \ar[drr]_{\varphi_i}&A_0 \ar@{..>}[dr]_<<<{\varphi_0}\\
&& Q}
\end{equation*}
for $i=1,2$. Let for any reducible tensor $x\otimes x' \in A_i
\otimes_k A_i$ $$\begin{array}{ll}(f_i, f_i)(x\otimes
x')=\overline x \otimes \overline x' \in A_0 \otimes_k A_0;&
\delta_0 (\overline x \otimes \overline x')=\overline x\cdot
\overline
x' \in A_0;\\
f_i(x\cdot x')=\overline x \otimes \overline x' \in A_0;&\delta_i
(x\otimes x')=x\cdot x'\in A_i.
 \end{array}$$
Since by  universality of the product $$A_0 \otimes_k A_0= (A_1
\otimes_k A_1)\otimes_{k[x_1,\dots, x_n]\otimes_k k[x'_1, \dots,
x'_m]} (A_2 \otimes_k A_2)$$ the homomorphism $\varphi'_0$ is
uniquely defined using homomorphisms
$$(\varphi_i, \varphi_i): A_i\otimes_k A_i \to Q: x\otimes
x'\mapsto \varphi_i(x)\cdot \varphi_i(x')
$$
as
$$\varphi'_0(\overline x \otimes \overline
x')=\varphi_i(x)\cdot \varphi_i(x') \in Q$$ then $\varphi_0:
A_0\to Q$ is also uniquely defined as $\varphi_0(\overline
y)=\varphi_i(y)$ for $\overline y=f_i(y)$, $y\in A_i$.

Now it rests to form product algebras for three different affine
representations \begin{equation*}\xymatrix{k[x_1, \dots, x_n]
\ar@{->>}[d]_{h_i}&\ar@{->>}[l]_<<<<\lambda k[x_1, \dots, x_n]
\otimes _k k[x_1', \dots, x_m'] \ar@{->>}[d]_{\overline
h_i}\ar@{->>}[r]^>>>>>\rho& k[x_1', \dots, x_m']
\ar@{->>}[d]^{h'_i}\\
A_i\ar@{=}[r]& A_i& \ar@{=}[l] A_i}
\end{equation*}
Horizontal homomorphisms are defined by correspondences $\lambda:
x'_l\mapsto 0$, $l=1,\dots, m$ and $\rho: x_j \mapsto 0$,
$j=1,\dots, n$ respectively.

Denoting by $\overline A$ the product algebra for $\overline h_i:
k[x_1, \dots, x_n] \otimes _k k[x_1', \dots,
x_m']\twoheadrightarrow A_i$, $i=0,1,2,$ we come to the
commutative diagram
\begin{equation}\label{cd2}\xymatrix{k[x_1, \dots, x_n] \ar@{->>}[d]_{h}&\ar@{->>}[l]_<<<<\lambda
k[x_1, \dots, x_n] \otimes _k k[x_1', \dots, x_m']
\ar@{->>}[d]_{\overline h}\ar@{->>}[r]^>>>>>\rho& k[x_1', \dots,
x_m']
\ar@{->>}[d]^{h'}\\
A&\ar@{->>}[l]_{p_A} \overline A\ar@{->>}[r]^{p'_A}&  A'}
\end{equation}
associated to closed immersions
\begin{eqnarray*}&&\xymatrix{&{\mathbb A}^n\;\; \ar@{^(->}[rr]&&\;\;\;
{\mathbb A}^n \times {\mathbb A}^m \;\;&&\ar@{_(->}[ll]
\;\; {\mathbb A}^m,&}\\
&&\xymatrix{\;\;(x_1,\dots,x_n)\ar@{|->}[r]&(x_1, \dots, x_n,
0,\dots,
0),}\\
&&\xymatrix{&&&\!\!(0,\dots, 0,x'_1,\dots,x'_m)&\ar@{|->}[l]
(x'_1, \dots, x'_m).}\end{eqnarray*} Vertical homomorphisms in
(\ref{cd2}) define affine representations for product algebras
built up according to the Easy case. Lower horizontal morphisms
are surjective by commutativity of the diagram (\ref{cd2}).

To prove isomorphicity of lower homomorphisms in (\ref{cd2})
consider sections of homo\-morphisms $k[x_1, \dots,
x_n]\stackrel{\lambda}{\twoheadleftarrow}k[x_1, \dots , x_n]
\otimes _k k[x'_1,\dots,
x'_m]\stackrel{\rho}{\twoheadrightarrow}k[x_1, \dots, x_m],$ i.e.
homomorph\-isms $k[x_1, \dots,
x_n]\stackrel{s_1}{\hookrightarrow}k[x_1, \dots , x_n] \otimes _k
k[x'_1,\dots, x'_m]\stackrel{s_2}{\hookleftarrow}k[x_1, \dots,
x_m]$ defined by following rules $s_1: f \mapsto f \otimes 1$,
$s_2: g \mapsto 1\otimes g$. A nonzero element from $A$ has
nonzero preimage in $k[x_1, \dots, x_n]$ which haves nonzero image
in $k[x_1, \dots, x_n] \otimes_k k[x'_1, \dots, x'_m]$. This image
is taken to nonzero element of at least one of algebras $A_1$,
$A_2$ and hence has nonzero image in $\overline A$.  This leads to
a homomorphism of $k$-algebras $s_A: A\to \overline A.$ It is a
section of the homomorphism $p_A$ by its construction.

Now remark that $s_A$ is surjective. Indeed, one can choose a
system of $k$-generators (which is not obliged to be a basis) in
$\overline A$ which are images of reducible tensors $f\otimes g$
from $k[x_1, \dots , x_n] \otimes _k k[x'_1,\dots, x'_m]$, $f\in
k[x_1, \dots, x_n ],$ $g\in k[x'_1, \dots, x'_m]$. This expression
is defined up to the action of $k^{\ast}=k\setminus 0$. A
polynomial $f$ has nonzero image  at least in one of algebras
$A_1, A_2$. Hence it is taken to nonzero element in $A$. Since
$p_A \circ s_A=\id_A$ then both $s_A$ and $p_A$ are isomorphisms.

Isomorphicity of $p'_A$ is proven analogously. \end{proof}

\subsection{Universality}
\begin{proposition} The construction of product algebra $A$ is indeed universal,
i.e. $A$ is true fibred product. \end{proposition}

\begin{proof} Let $A_T$ be an affine  $k$-algebra together with two
homomorphisms $\chi_i: A_T \to A_i$, $i=1,2$ such that $\chi_1
\circ f_1=\chi_2\circ f_2$. We construct a homomorphism $\varphi$
to complete the diagram \begin{equation*}
\xymatrix{&\ar[ld]_{\chi_1}A_T \ar@{..>}[d]^{\varphi} \ar[rd]^{\chi_2}\\
A_1& \ar[l]_{\overline f_1} A \ar[r]^{\overline f_2}& A_2}
\end{equation*}
Choose  appropriate affine representations of algebras $A_T, A,
A_i$, $i=1,2$. Perform the manipulations as described in the Easy
case. Namely, let $J$ be an ideal in $k[x_1, \dots, x_n]$
generated by all monomials taken to 0 in $A_T$, $A$ and $A_i$,
$i=1,2$.

Quotient algebra $k[x_1, \dots, x_n]/J$ includes in the
commutative diagram
\begin{equation*}\xymatrix{&\ar@{->>}[ld]_{f_1}A_1\\
A_0&A_T \ar[u]^{\chi_1}
\ar[d]_{\chi_2}&\ar[ul]_<<<<{\!\!\!\overline f_1}
\ar[ld]^<<<{\!\!\!\overline f_2} A&\ar@{->>}[l] \ar@{->>}[ull]
\ar@{->>}[lld] k[x_1, \dots , x_n]/J\\
&\ar@{->>}[ul]^{f_2} A_2}
\end{equation*}

Choose an arbitrary element $\alpha \in A_T$; it is taken to
$\chi_1(\alpha) \in A_1$ and to $\chi_2(\alpha)\in A_2$ so that
$f_1 \chi_1(\alpha) =f_2 \chi_2(\alpha).$ Any $\gamma$ of
preimages of $\alpha $ in $k[x_1, \dots , x_n]/J$ is also taken to
$f_1 \chi_1(\alpha)=f_2 \chi_2(\alpha) \in A_0$ and to some
element $\overline \alpha \in A$. We claim that $\overline \alpha$
does not depend on the choice of $\gamma$.

Choose another $\gamma' \in k[x_1, \dots , x_n]/J$ taken to
$\alpha$, then $\gamma -\gamma'$ maps to zero in $A_T$ and hence
it is mapped to zero in both $A_i,$ $i=1,2$. Then by the
construction of $A$ (cf. Easy case) $\gamma -\gamma'$ is taken to
zero in $A$. This shows that there is a homomorphism $\varphi:A_T
\to A$ such that $\chi_i=\varphi \circ \overline f_i$, $i=1,2,$ as
required.

It rests to confirm ourselves that the homomorphism $\varphi: A_T
\to A$ does not depend on the choice of affine representations of
algebras  $A_T$, $A_i$, $i=0,1,2.$ For this sake assume that there
are two different collections of affine representations include
into commutative diagrams
\begin{equation}\label{tworep}\xymatrix{&\ar@{->>}[ldd]_{h_1} k[x_1, \dots, x_n] \ar@{->>}[d]^{h_T} \ar@{->>}[rdd]^{h_2}\\
&\ar[ld] A_T \ar[rd]&\\
A_1\ar@{->>}[dr]_{f_1}&\ar[l]A\ar[r]&\ar@{->>}[dl]^{f_2}A_2\\
&A_0&}\quad \xymatrix{&\ar@{->>}[ldd]_{h'_1} k[x'_1, \dots, x'_m] \ar@{->>}[d]^{h'_T} \ar@{->>}[rdd]^{h'_2}\\
&\ar[ld] A_T \ar[rd]&\\
A_1\ar@{->>}[dr]_{f_1}&\ar[l]A\ar[r]&\ar@{->>}[dl]^{f_2}A_2\\
&A_0&}
\end{equation}

Then consider the algebra $k[x_1,\dots,x_n]\otimes_k k[x'_1,
\dots, x'_m]$ and its homomorphisms on quotient algebras
$k[x_1,\dots,x_n]\stackrel{\lambda}{\longleftarrow}k[x_1,\dots,x_n]\otimes_k
k[x'_1, \dots, x'_m]\stackrel{\rho }{\longrightarrow}k[x'_1,
\dots, x'_m]$ which are defined by their kernels  $\ker
\lambda=(x'_1, \dots, x'_m)$, $\ker \rho=(x_1, \dots, x_n)$ and
correspond to immersions of subspaces ${\mathbb A}^n
\stackrel{\lambda^{\sharp}}{\hookrightarrow} {\mathbb A}^n \times
{\mathbb A}^m \stackrel{\rho^{\sharp}}{\hookleftarrow}{\mathbb
A}^m$ as linear subvarieties. Representations  (\ref{tworep})
define induced affine representations for algebras $A_T$, $A_i$,
$i=0,1,2,$ according to the following rule: $\overline
h_T(x_j\otimes x'_l)=h_T(x_j)\cdot h'_T(x'_l)$, $\overline
h_i(x_j\otimes x'_l)=h_i(x_j)\cdot h'_i(x'_l)$, $i=0,1,2$,
$j=1,\dots,n,$ $l=1,\dots,m.$ This rule corresponds to the
composites of immersions of prime spectra
\begin{eqnarray*}\overline h_T^{\sharp}&:&\Spec A_T
\stackrel{\diag}{\hookrightarrow} \Spec A_T \times \Spec A_T
\stackrel{(h^{\sharp}_T,{h'_T}^{\sharp})}{\hookrightarrow}
{\mathbb A}^n \times {\mathbb A}^m,
\\
\overline h_i^{\sharp}&:&\Spec A_i
\stackrel{\diag}{\hookrightarrow} \Spec A_i \times \Spec A_i
\stackrel{(h^{\sharp}_i,{h'_i}^{\sharp})}{\hookrightarrow}{\mathbb
A}^n \times {\mathbb A}^m. \end{eqnarray*}

Let $J$ be the ideal generated by all monomials in $k[x_1, \dots ,
x_n]$ taken to 0 in $A_T$ and in $A_i$, $i=0,1,2$, $J'$ be the
similar ideal in $k[x'_1, \dots, x'_m],$ and $\overline J$ the
similar ideal in $k[x_1, \dots , x_n]\otimes _k k[x'_1, \dots ,
x'_m].$ Quotient algebras by these ideals include into the
following commutative diagram
\begin{equation*}\xymatrix{k[x_1,\dots,x_n]\ar@{->>}[d]
&\ar[l]_<<<<<<{\lambda}k[x_1,\dots,x_n]\otimes_k k[x'_1, \dots,
x'_m]\ar@{->>}[d] \ar[r]^>>>>>>{\rho }&k[x'_1, \dots,
x'_m]\ar@{->>}[d]\\
k[x_1, \dots , x_n]/J&\ar@{->>}[l]k[x_1, \dots , x_n]\otimes_k
k[x'_1, \dots, x'_m]/\overline J \ar@{->>}[r]&k[x'_1, \dots ,
x'_m]/J'}\end{equation*} Now we perform the construction of
homomorphisms $\varphi, \varphi'$ and $\overline \varphi$ using
 quotient algebras $k[x_1, \dots , x_n]/J, k[x'_1, \dots ,
x'_m]/J'$ and $k[x_1, \dots , x_n]\otimes_k k[x'_1, \dots,
x'_m]/\overline J $ respectively, as described in the beginning of
this subsection. This yields in two commutative diagrams
\begin{equation*}\xymatrix{k[x_1, \dots , x_n]/J\ar@{->>}[dd] \ar@{->>}[dr]&&\ar@{->>}[ll]k[x_1, \dots , x_n]\otimes_k k[x'_1,
\dots,\! x'_m]/\overline J  \ar@{->>}[dd]
\ar@{->>}[dr]\\
&\ar@{->>}[ld]_{\varphi}A_T\ar@{=}[rr]&&\ar@{->>}[ld]_{\overline \varphi}A_T\\
A\ar@{=}[rr]&&A}\end{equation*}
\begin{equation*}\xymatrix{k[x_1, \dots , x_n]\otimes_k k[x'_1,
\dots, x'_m]/\overline J  \ar@{->>}[dd]
\ar@{->>}[dr]\ar@{->>}[rr]&&k[x'_1, \dots , x'_m]/J'\ar@{->>}[dd]
\ar@{->>}[dr]\\
&\ar@{->>}[ld]_{\overline \varphi}A_T\ar@{=}[rr]&&\ar@{->>}[ld]_{\varphi'}A_T\\
A\ar@{=}[rr]&&A}\end{equation*} what implies that
$\varphi=\varphi'=\overline \varphi.$ \end{proof}

\begin{example} For example consider  $A_1=A_2=k[x,y]/(x^2-y),$ $A_0=k[x]/(x^2)$,
with homomorphisms $f_i: x\mapsto x, y\mapsto 0,$ $i=1,2.$ The
"occasional"\, affine representation used in the definition of
these algebras cannot be involved to construct their product since
$A_1 \otimes_{k[x,y]} A_2=k[x,y]/(x^2-y)\ne A_0.$ Then it is
necessary to form a new affine representation which fits for
constructing a product, using Hard case.

The algebra $A_0=k[x]/(x^2)$ contains unique maximal principal
ideal. Its generating element $x$ corresponds to first variable
 $x_1$ and gives rise to the affine representation $k[x_1]\twoheadrightarrow A_0,$ $x_1\mapsto x$.
Further,  manipulations according to the Hard case lead to the
polynomial algebra  $k[x_1, x_2, x_3, y]$ and representations for
algebras $$A_1=k[x_1, x_2, x_3, y]/(x_1^2-x_2, x_3, x_1^2-y),\quad
A_2=k[x_1, x_2, x_3, y]/(x_1^2-x_3, x_2, y).$$ These algebras
represent geometrically two parabolas with common tangent line in
two 2-dimensional planes in 4-dimensional affine space. Planes
meet along this tangent line. In this case
\begin{eqnarray*} A_1 \otimes_{k[x_1, x_2, x_3, y]} A_2&=&k[x_1,
x_2, x_3,
y]/(x_1^2-x_2, x_3, x_1^2-y,x_1^2-x_3, x_2, y)\\
&=&k[x_1, x_2, x_3, y]/(x_1^2, x_2,x_3,y)=A_0.\end{eqnarray*} This
means that our two parabolas intersect along the subscheme defined
by the algebra $A_0$. This validates application of Easy case to
affine representations we've constructed. The list $J$ of
monomials vanishing in both algebras $A_i,$ $i=1,2,$ is empty. The
ideal $<L>$ has a form $<L>=(x_1^2-x_2-x_3, x_2-y, x_2x_3)$, and
$A_1\times_{A_0}A_2=k[x_1, x_2,
x_3,y]/<L>=k[x_1,x_2,x_3,y]/(x_1^2-x_2-x_3, x_2-y,
x_2x_3)=k[x_1,x_2,x_3]/(x_1^2-x_2-x_3, x_2x_3).$ From the
geometrical point of view this is union of two parabolas having a
common tangent line at the origin. Parabolas lie in different
2-planes which meet along parabolas' common tangent line.
 \end{example}

\end{document}